\documentclass[graybox]{svmult}

\usepackage{mathptmx}       
\usepackage{helvet}         
\usepackage{courier}        
\usepackage{type1cm}        
\usepackage{makeidx}         
\usepackage{graphicx}        
\usepackage{multicol}        
\usepackage[bottom]{footmisc}
\usepackage{amssymb,amsmath}

\makeindex             

\def\hangbox to #1 #2{\vskip1pt\hangindent #1\noindent \hbox to #1{#2}$\!\!$}

\begin{document}

\title*{An asymptotic equivalence between two frame perturbation theorems}

\author{B. A. Bailey}

\institute{B. A. Bailey \at Department of Mathematics, Texas A\&M University, College Station, TX 77843-3368, \email{abailey@math.tamu.edu}}

\maketitle

\abstract{
In this paper, two stability results regarding exponential frames are compared.  The theorems, (one proven herein, and the other in \cite{SZ}), each give a constant such that if $\sup_{n \in \mathbb{Z^d}}\| \epsilon_n \|_\infty < C$, and $(e^{i \langle \cdot , t_n \rangle})_{n \in \mathbb{Z}^d}$ is a frame for $L_2[-\pi,\pi]^d$, then $(e^{i \langle \cdot , t_n +\epsilon_n \rangle})_{n \in \mathbb{Z}^d}$ is a frame for $L_2[-\pi,\pi]^d$.  These two constants are shown to be asymptotically equivalent for large values of $d$.}

\section{The perturbation theorems}\label{S:1}

\noindent We define a frame for a separable Hilbert space $H$ to be a sequence $(f_n)_n \subset H$ such that for some $0<A \leq B,$
\begin{equation}
A^2 \| f \|^2 \leq \sum_n  |\langle f , f_n \rangle|^2 \leq B^2 \| f \|^2, \quad f \in H.\nonumber
\end{equation}
The best $A^2$ and $B^2$ satisfying the inequality above are said to be the frame bounds for the frame.  If $(e_n)_n$ is an orthonormal basis for $H$, the synthesis operator $L e_n = f_n$ is bounded, linear, and onto, iff $(f_n)_n$ is a frame.  Equivalently, $(f_n)_n$ is a frame iff the operator $L^*$ is an isomorphic embedding, (see \cite{CA}).  In this case, $A$ and $B$ are the best constants such that 
$$A\| f \| \leq \| L^* f \| \leq B \|f\|, \quad f \in H.$$

\noindent The simplest stability result regarding exponential frames for $L_2[-\pi,\pi]$ is the theorem below, which follows immediately from \cite[Theorem 13, p 160]{Y}.

\begin{theorem}\label{simplest}
Let $(t_n)_{n \in \mathbb{Z}} \subset \mathbb{R}$ be a sequence such that $(h_n)_{n \in \mathbb{Z}} :=  \big(\frac{1}{\sqrt{2\pi}}e^{i t_n x}\big)_{n \in \mathbb{Z}}$ is a frame for $L_2[-\pi,\pi]$ with frame bounds $A^2$ and $B^2$.  If $(\tau_n)_{n \in \mathbb{Z}} \subset \mathbb{R}$ and $(f_n)_{n \in \mathbb{Z}} := \big(\frac{1}{\sqrt{2\pi}} e^{i \tau_n x}\big)_{n \in \mathbb{Z}}$ is a sequence such that
\begin{equation}
\sup_{n \in \mathbb{Z}}| \tau_n - t_n | < \frac{1}{\pi}\ln\left(1+\frac{A}{B}\right),
\end{equation}
then the sequence $(f_n)_{n \in \mathbb{Z}}$ is also a frame for $L_2[-\pi,\pi]$.
\end{theorem}

\noindent The following theorem is a very natural generalization of Theorem \ref{simplest} to higher dimensions.

\begin{theorem}\label{myframepert}
Let $(t_k)_{k \in \mathbb{N}} \subset \mathbb{R}^d$ be a sequence such that $(h_k)_{k \in \mathbb{N}} :=  \big(\frac{1}{(2\pi)^{d/2}}e^{\langle (\cdot) , t_k \rangle}\big)_{k \in \mathbb{N}}$ is a frame for $L_2[-\pi,\pi]^d$ with frame bounds $A^2$ and $B^2$.  If $(\tau_k)_{k \in \mathbb{N}} \subset \mathbb{R}^d$ and $(f_k)_{k \in \mathbb{N}} := \big(\frac{1}{(2\pi)^{d/2}} e^{i \langle (\cdot) , \tau_k \rangle}\big)_{k \in \mathbb{N}}$ is a sequence such that
\begin{equation}
\sup_{k \in \mathbb{N}} \Arrowvert \tau_k - t_k \Arrowvert_\infty < \frac{1}{\pi d}\ln\left(1+\frac{A}{B}\right),
\end{equation}
then the sequence $(f_k)_{k \in \mathbb{N}}$ is also a frame for $L_2[-\pi,\pi]^d$.
\end{theorem}

\noindent The proof of Theorem \ref{myframepert} relies on the following lemma:

\begin{lemma}\label{mainlemma}
Choose $(t_k)_{k \in \mathbb{N}} \subset \mathbb{R}^d$ such that $(h_k)_{k \in \mathbb{N}} :=  \big(\frac{1}{(2\pi)^{d/2}}e^{\langle (\cdot) , t_k \rangle}\big)_{k \in \mathbb{N}}$ satisfies $$\Big\| \sum_{k=1}^n a_k h_k \Big\|_{L_2 [-\pi,\pi]^d} \leq B \Big(\sum_{k=1}^n |a_k|^2 \Big)^{1/2}, \quad \mathrm{for \ all} \quad (a_k)_{k=1}^n \subset \mathbb{C}.$$  If $(\tau_k)_{k \in \mathbb{N}} \subset \mathbb{R}^d$, and $(f_k)_{k \in \mathbb{N}} := \big(\frac{1}{(2\pi)^{d/2}} e^{i \langle (\cdot) , \tau_k \rangle}\big)_{k \in \mathbb{N}}$, then for all $r,s\geq 1$ and any finite sequence $(a_k)_k$, we have
\begin{equation}
\Bigg\Arrowvert \sum_{k=r}^s a_k ( h_k - f_k ) \Bigg\Arrowvert_{L_2 [-\pi,\pi]^d} \\
{} \leq B\Big( e^{\pi d \big(\sup\limits_{r \leq k \leq s} \Arrowvert \tau_k-t_k \Arrowvert_\infty\big)}-1 \Big) \Big(\sum_{k=r}^s |a_k|^2\Big)^\frac{1}{2}.\nonumber
\end{equation}
\end{lemma}

\noindent This lemma is a slight generalization of Lemma 5.3, proven in \cite{BB} using simple estimates.  Lemma \ref{mainlemma} is proven similarly.  Now for the proof of Theorem \ref{myframepert}.

\begin{proof}
Define $\delta = \sup_{k \in \mathbb{N}} \Arrowvert \tau_k - t_k \Arrowvert_\infty$.  Lemma \ref{mainlemma} shows that the map $\tilde{L} e_n = f_n$ is bounded and linear, and that
\begin{equation}
\| L - \tilde{L}\| \leq B\big( e^{\pi d \delta} -1 \big) := \beta A\nonumber 
\end{equation}
for some $0 \leq \beta <1$.  This implies
\begin{equation}
\|L^*f-\tilde{L}^*f\| \leq \beta A, \quad \mathrm{when} \quad \|f\|=1.
\end{equation}
Rearranging, we have $$A(1-\beta) \leq \|\tilde{L}^*f\|, \quad \mathrm{when} \quad \|f\|=1.$$  By the previous remarks regarding frames,  $(f_k)_{k \in \mathbb{N}}$ is a frame for $L_2[-\pi,\pi]^d$.
\end{proof}

\noindent Theorem \ref{sunzhou}, proven in \cite{SZ}, is a more delicate frame perturbation result with a more complex proof:\\

\begin{theorem}\label{sunzhou}
Let $(t_k)_{k \in \mathbb{N}} \subset \mathbb{R}^d$ be a sequence such that $(h_k)_{k \in \mathbb{N}} :=  \big(\frac{1}{(2\pi)^{d/2}}e^{\langle (\cdot) , t_k \rangle}\big)_{k \in \mathbb{N}}$ is a frame for $L_2[-\pi,\pi]^d$ with frame bounds $A^2$ and $B^2$.  For $d \geq 1$, define $$D_d(x) := \Big(1-\cos \pi x+\sin \pi x  +\frac{\sin \pi x}{\pi x} \Big)^d -\Big( \frac{\sin \pi x}{\pi x} \Big)^d,$$ and let $x_d$ be the unique number such that $0 < x_d \leq 1/4$ and $D_d(x_d)=\frac{A}{B}$.  
If $(\tau_k)_{k \in \mathbb{N}} \subset \mathbb{R}^d$ and $(f_k)_{k \in \mathbb{N}} := \big(\frac{1}{(2\pi)^{d/2}} e^{i \langle (\cdot) , \tau_k \rangle}\big)_{k \in \mathbb{N}}$ is a sequence such that
\begin{equation}\label{ineq}
\sup_{k \in \mathbb{N}} \| \tau_k - t_k \|_\infty< x_d,
\end{equation}
then the sequence $(f_k)_{k \in \mathbb{N}}$ is also a frame for $L_2[-\pi,\pi]^d$.
\end{theorem}

\section{An asymptotic equivalence}

\noindent It is natural to ask how the constants $x_d$ and $\frac{1}{\pi d}\ln\big(1+\frac{A}{B}\big)$ are related.  Such a relationship is given in the following theorem.

\begin{theorem}\label{main4}
If $x_d$ is the unique number satisfying $0 < x_d <1/4$ and $D_d(x_d)=\frac{A}{B}$, then
$$\lim_{d \rightarrow \infty} \frac{x_d - \frac{1}{\pi d}\ln\big(1+\frac{A}{B}\big)}{\frac{\big[ \ln\big(1+\frac{A}{B} \big) \big]^2}{6\pi\big(1+\frac{B}{A}\big)d^2}} =1. $$
\end{theorem}

\noindent We prove the theorem with a sequence of propositions.

\begin{proposition}\label{first1}
\noindent Let $d$ be a positive integer.  If
\begin{eqnarray}
f(x)  &:= & 1-\cos(x)+\sin(x)+\mathrm{sinc}(x),\nonumber\\
g(x)  &:= & \mathrm{sinc}(x)\nonumber,
\end{eqnarray}
then 
\begin{eqnarray}
& 1) & \quad f'(x)+g'(x) > 0, \quad x \in (0,\pi/4),\nonumber\\
& 2) & \quad g'(x) < 0, \quad x \in (0,\pi/4),\nonumber \\
& 3) & \quad f''(x) > 0, \quad x \in (0, \Delta) \quad \mathrm{for \ some} \quad 0<\Delta < 1/4\nonumber.
\end{eqnarray}
\end{proposition}

\noindent The proof of Proposition \ref{first1} involves only elementary calculus and is omitted.




\begin{proposition}\label{convex}
The following statements hold: \\
\noindent 1)  For $d>0$, $D_d(x)$ and $D'_d(x)$ are positive on $(0,1/4)$.\\
\noindent 2)  For all $d>0$, $D''_d(x)$ is positive on $(0,\Delta)$.
\end{proposition}

\begin{proof}
Note $D_d(x) = f(\pi x)^d-g(\pi x)^d$ is positive.  This expression yields $$D'_d(x)/(d\pi) =f(\pi x)^{d-1}f'(\pi x)-g(\pi x)^{d-1}g'(\pi x)>0 \quad \mathrm{on} \quad (0,1/4)$$ by Proposition \ref{first1}. Differentiating again, we obtain
\begin{eqnarray}
D''_d(x)/(d\pi^2) & = &(d-1)\big[ f(\pi x)^{d-2}(f'(\pi x))^2- g(\pi x)^{d-2} (g'(\pi x))^2 \big] +\nonumber\\ 
 &+& [f(\pi x)^{d-1} f''(\pi x)- g(\pi x)^{d-1} g''(\pi x)] \quad \text{on} \quad (0,1/4).\nonumber
\end{eqnarray}
If $ g''(\pi x) \leq 0$ for some $x \in (0, 1/4)$, then the second bracketted term is positive.  If $ g''(\pi x)>0$ for some $x \in (0, 1/4)$, then the second bracketted term is positive if $f''(\pi x)-g''(\pi x) >0$, but $$f''(\pi x)-g''(\pi x)=\pi^2(\cos(\pi x)-\sin(\pi x))$$ is positive on $(0,1/4)$.

\noindent To show the first bracketted term is positive, it suffices to show that $$f'(\pi x)^2 > g'(\pi x)^2 = (f'(\pi x)+g'(\pi x))(f'(\pi x)-g'(\pi x))>0$$ on $(0, \Delta)$.  Noting $f'(\pi x)-g'(\pi x)=\pi (\cos(\pi x)+\sin(\pi x))>0$, it suffices to show that $f'(\pi x)+g'(\pi x)>0$, but this is true by Proposition \ref{first1}.
\end{proof}

\noindent Note that Proposition \ref{convex} implies $x_d$ is unique.

\begin{corollary}
We have  $\lim_{d \rightarrow \infty} x_d = 0$.
\end{corollary}

\begin{proof}
Fix $n>0$ with $1/n <\Delta$, then $\lim_{d \rightarrow \infty} D_d(1/n) = \infty$ (since $f$ increasing implies $0<-\cos(\pi /n)+\sin(\pi/n)+\mathrm{sinc}(\pi/n)).$  For sufficiently large $d$, $D_d(1/n)>\frac{A}{B}$.  But $\frac{A}{B} = D_d(x_d) < D_d(1/n)$, so $x_d <1/n$ by Proposition \ref{convex}.
\end{proof}

\begin{proposition}\label{mainprop}
Define $\omega_d = \frac{1}{\pi d}\ln\big(1+\frac{A}{B}\big)$. We have
\begin{eqnarray}
\lim_{d \rightarrow \infty} d\Big(\frac{A}{B}-D_d(\omega_d)\Big) = \frac{A}{6B}\Big[\ln\Big(1+\frac{A}{B}\Big)\Big]^2,\nonumber\\
\lim_{d \rightarrow \infty} \frac{1}{d} D'_d(\omega_d) = \pi\Big(1+\frac{A}{B}\Big),\nonumber\\
\lim_{d \rightarrow \infty} \frac{1}{d}D'_d (x_d) =  \pi\Big(1+\frac{A}{B}\Big).\nonumber
\end{eqnarray}
\end{proposition}

\begin{proof}
1) For the first equality, note that
\begin{equation}\label{Ddef}
D_d (\omega_d)  =  \Big[(1+h(x))^{\ln(c)/x}-g(x)^{\ln(c)/x}\Big]\Big|_{x=\frac{\ln(c)}{d}}
\end{equation}
where $h(x)=-\cos(x)+\sin(x)+\mathrm{sinc}(x)$, $g(x)=\mathrm{sinc}(x)$, and $c=1+\frac{A}{B}$.  L'Hospital's rule implies that $$\lim_{x \rightarrow 0} (1+h(x))^{\ln(c)/x} =c \quad \mathrm{and} \quad \lim_{x \rightarrow 0} g(x)^{\ln(c)/x}=1.$$  Looking at the first equality in the line above, another application of L'Hospital's rule yields 
\begin{equation}\label{interm}
\lim_{x \rightarrow 0} \frac{(1+h(x))^{\ln(c)/x} -c}{x} = c\ln(c) \Bigg[\frac{\frac{h'(x)}{1+h(x)}-1}{x}-\frac{\ln(1+h(x))-x}{x^2}  \Bigg].
\end{equation}
Observing that $h(x) = x+x^2/3+O(x^3))$, we see that $$ \lim_{x \rightarrow 0} \frac{\frac{h'(x)}{1+h(x)}-1}{x} = -\frac{1}{3}. $$  L'Hospital's rule applied to the second term on the right hand side of equation (\ref{interm}) gives
\begin{equation}\label{hpart}
\lim_{x \rightarrow 0} \frac{(1+h(x))^{\ln(c)/x} -c}{x} = \frac{-c\ln(c)}{6}.
\end{equation}
In a similar fashion,
\begin{equation}\label{interm2}
\lim_{x\rightarrow 0} \frac{g(x)^{\ln(c)/x}-1}{x} = \ln(c) \lim_{x\rightarrow 0} \Bigg[ \frac{\frac{g'(x)}{g(x)}}{x} - \frac{\ln(g(x))}{x^2}\Bigg].
\end{equation}
Observing that $g(x)=1-x^2/6+O(x^4)$, we see that $$ \lim_{x \rightarrow 0} \frac{\frac{g'(x)}{g(x)}}{x} =-\frac{1}{3}.$$  L'Hospital's rule applied to the second term on the right hand side of equation (\ref{interm2}) gives
\begin{equation}\label{gpart}
\lim_{x\rightarrow 0} \frac{g(x)^{\ln(c)/x}-1}{x} = -\frac{\ln(c)}{6}.
\end{equation}
Combining equations (\ref{Ddef}) (\ref{hpart}), and (\ref{gpart}), we obtain 
\begin{equation}
\lim_{d \rightarrow \infty} d\Big(\frac{A}{B}-D_d(\omega_d)\Big) = \frac{A}{6B}\Big[\ln\Big(1+\frac{A}{B}\Big)\Big]^2\nonumber.
\end{equation}

\noindent 2) For the second equality we have, (after simplification),
\begin{equation}
\frac{1}{d}D'_d(\omega_d) = \pi\Bigg[ \frac{\Big(1+h\big( \frac{\ln(c)}{d} \big)\Big)^{\big(\ln(c)\big)/\big(\frac{\ln(c)}{d}\big)}}{1+h\Big(\frac{\ln(c)}{d}\Big)}-\frac{g\Big(\frac{\ln(c)}{d} \Big)^{\big(\ln(c)\big)/\big( \frac{\ln(c)}{d}\big)}}{g\Big(\frac{\ln(c)}{d} \Big)}g'\Big(\frac{\ln(c)}{d} \Big)\Bigg].\nonumber
\end{equation}
In light of the previous work, this yields
\begin{equation}
\lim_{d \rightarrow \infty} \frac{1}{d} D'_d(\omega_d) = \pi\Big(1+\frac{A}{B}\Big) . \nonumber
\end{equation}
3) To derive the third equality, note that  $(1+h(\pi x_d))^d = \frac{A}{B} + g(\pi x_d)^d$ yields
\begin{equation}\label{xstuff}
\frac{1}{d}D'_d(x_d) = \pi \Bigg[ \frac{\frac{A}{B} + g(\pi x_d)^d}{1+h(\pi x_d)}h'(\pi x_d)-\frac{g(\pi x_d)^d}{g(\pi x)}g'(\pi x_d)  \Bigg].
\end{equation}
Also, the first inequality in propostion \ref{mainprop} yields that, for sufficiently large $d$ (also large enough so that $x_d < \Delta$ and $\omega_d <\Delta$), that $D_d(\omega_d) < \frac{A}{B} =D_d(x_d)$.  This implies $\omega_d < x_d$ since $D_d$ is increasing on $(0, 1/4)$. But $D_d$ is also convex on $(0, \Delta)$, so we can conclude that 
\begin{equation}\label{deriv}
D'_d(\omega_d) < D'_d(x_d).
\end{equation}
Combining this with equation (\ref{xstuff}), we obtain
\begin{eqnarray}
\Bigg[\frac{1}{d} D'_d(\omega_d) +\frac{\pi g(\pi x_d)^d}{g(\pi x_d)}g'(\pi x_d) \Bigg]\Big(\frac{1+h(\pi x_d)}{h'(\pi x_d)}\Big) < \pi\Big(\frac{A}{B}+g(\pi x_d)^d\Big) < \pi\Big(1+\frac{A}{B} \Big).\nonumber
\end{eqnarray}
The limit as $d \rightarrow \infty$ of the left hand side of the above inequality is $\pi\Big(1+\frac{A}{B} \Big)$, so $$\lim_{d \rightarrow \infty} \pi \Big(\frac{A}{B}+ g(\pi x_d)^d\Big)= \pi\Big(1+\frac{A}{B} \Big).$$  Combining this with equation (\ref{xstuff}), we obtain $$ \lim_{d \rightarrow \infty} \frac{1}{d}D'_d(x_d) = \pi\Big(1+\frac{A}{B} \Big).$$
\end{proof}

\noindent Now we complete the proof of Theorem \ref{main4}.

\noindent For large $d$, the mean value theorem implies
\begin{eqnarray}
 \frac{D_d(x_d)-D_d(\omega_d)}{x_d-\omega_d} = D'_d(\xi), \quad \xi \in (\omega_d,x_d)\nonumber,
\end{eqnarray}
so that
\begin{eqnarray}
x_d-\omega_d = \frac{\frac{A}{B}-D_d(\omega_d)}{D'_d(\xi)}\nonumber.
\end{eqnarray}
For large $d$, convexity of $D_d$ on $(0, \Delta)$ implies
\begin{equation}
\frac{d\Big(\frac{A}{B}-D_d(\omega_d)\Big)}{\frac{1}{d}D'_d (x_d)} < d^2 (x_d-\omega_d) < \frac{d\Big(\frac{A}{B}-D_d(\omega_d)\Big)}{\frac{1}{d}D'_d(\omega_d)}\nonumber.
\end{equation}
Applying Proposition \ref{mainprop} proves the theorem.

\begin{acknowledgement}
This research was supported in part by the NSF Grant DMS0856148.
\end{acknowledgement}

\begin{thebibliography}{99.}
\bibitem{BB} Bailey, B.A.: Sampling and recovery of multidimensional bandlimited functions via frames. J. Math. Anal. Appl. \textbf{367}, Issue 2 374--388 (2010) 
\bibitem{CA} Casazza, P.G.: The art of frames. Taiwanese J. Math. \textbf{4}. No. 2  129--201 (2001)
\bibitem{SZ} Sun, W., Zhou, X.: On the stability of multivariate trigonometric systems. J. Math. Anal. Appl. \textbf{235},  159--167 (1999)
\bibitem{Y} Young, R.M.: An Introduction to Nonharmonic Fourier Series. Academic Press (2001)
\end{thebibliography}
\end{document}